\documentclass{article}
\usepackage{aosformat}
\usepackage{authblk}

\usepackage{natbib}
 \bibpunct[, ]{(}{)}{,}{a}{}{,}%

 \usepackage{caption}
\usepackage[labelfont=sf]{subcaption}
 \newcommand\blfootnote[1]{%
  \begingroup
  \renewcommand\thefootnote{}\footnote{#1}%
  \addtocounter{footnote}{-1}%
  \endgroup
}

\title{Extracting Alternative Solutions from Benders Decomposition}
\author[1,2]{Matthew Viens\thanks{mviens@wisc.edu}}
\author[2]{William E. Hart\thanks{wehart@sandia.gov}}
\author[1]{Michael Ferris\thanks{ferris@cs.wisc.edu}}
\affil[1]{Department of Computer Sciences,
University of Wisconsin-Madison}
\affil[2]{Center for Computing Research,
Sandia National Laboratories}

\date{September 2025}

\begin{document}
\maketitle
\begin{abstract}
    \blfootnote{Funding: This work was supported in part by the Laboratory Directed Research and Development program at Sandia
National Laboratories, a multimission laboratory managed and operated by National Technology and
Engineering Solutions of Sandia LLC, a wholly owned subsidiary of Honeywell International Inc. for the U.S.
Department of Energy’s National Nuclear Security Administration under contract DE-NA0003525.	Additional funding came from the Morgridge Chair of Computer Sciences at University of Wisconsin-Madison. This research was performed while some authors were participating in the Architecture of Green Energy Systems Program hosted by the University of Chicago, which is supported by the National
Science Foundation (Grant No. DMS-1929348).}
    \blfootnote{Acknowledgements:  We would like to express our sincere gratitude to Cynthia Phillips for her invaluable contributions to this research. We are also grateful to Jared Gearhart for their support and assistance throughout the course of this work.}
    
We show how to extract alternative solutions for optimization problems solved by \nameBenders. In practice, alternative solutions provide useful insights for complex applications; some solvers do support generation of alternative solutions but none appear to support such generation when using \nameBenders. We propose a new post-processing method that extracts multiple optimal and near-optimal solutions using the cut-pool generated during \nameBenders. Further, we provide a geometric framework for understanding how the adaptive approximation in \nameBenders\  relates to alternative solutions. We demonstrate this technique on stochastic programming and interdiction modeling, and we highlight use cases that require the ability to enumerate all optimal solutions.
\end{abstract}

\section{Introduction}
Optimization solvers traditionally return a single optimal solution. We know that multiple solutions naturally occur in a variety of applications, but computational tools to generate alternative solutions remain limited. However, alternative solutions often provide practical value. For example, end-users may have secondary concerns, like unexpressed objectives or modeling uncertainty, that motivate an analysis of alternative solutions \citep{brill_1979_public_planning_alternatives}. There is a growing body of literature that describes methods to generate and diversify alternative solutions \citep[e.g.][]{Lau_2024_mga_alternatives_energy_applications_review, Petit2019_alternative_solutions_as_solution_engineering_review_ijoc, Ahanor_Trapp_DiversiTree_IJOC}, and commercial solvers have begun to integrate this functionality (e.g. Gurobi, CPLEX). However, previous work has not considered decomposition methods that generate alternative solutions. This treatment of more specific algorithms has not been helped by the scattered and diverse names used to describe several solutions to an optimization problem, especially in treating both exact optimal or near optimal solutions; we standardize on the term alternative solutions for this concept. We show how to adapt \nameBenders\ \citep{Benders1962, l_shaped_slyke} to identify alternative solutions. We focus on, and we describe stochastic programming and interdiction applications for our alternative solution exemplars.

We propose a new post-processing method that extracts alternative optimal and near-optimal solutions using the cut-pool generated during \nameBenders. Our overall approach involves three steps. We solve a \nameBenders\  problem to optimality. Then we use the generated cut-pool in our post-processing method to generate candidate alternative solutions. Finally, we filter the candidate alternative solutions with a certification step. We call this approach AOS-Benders (alternative optimal solutions for Benders). 

We motivate AOS-Benders by describing an epigrapical and level-set view of the connection between alternative solutions in the original problem and the \nameBenders\ problem. \nameBenders\ separates into two problems (master and subproblem) that have a corresponding split in variables matching the first-stage and second-stage variables from stochastic programming; we use this first-stage and second-stage variable convention to describe our variable distinction.
We describe two core theoretical results. First, the set of alternative solutions for the \nameBM\  problem contains the set of alternative solutions for the first-stage variables. Second, the converse is generally not true. Alternative solutions in the second-stage variables are also possible. Our framework also treats the differing requirements of generating exact and approximate alternative solutions, and it provides methods for generating both kinds.

The paper is laid out as follows. In Section \ref{sec:background:aos}, we review alternative solution theory and generation methods. In Section \ref{sec:background:benders}, we review \nameBenders. We develop a theoretical analysis and describe AOS-Benders in Section \ref{sec:ext_benders}. We present the Farmer's Problem as a stochastic programming exemplar in Section \ref{sec:farmer}, and we present the s-t Shortest Path Problem as an interdiction exemplar in Section \ref{sec:israeli_wood}. We conclude with a discussion of results and future work in Section \ref{sec:conclusions}. For formatting of variables and data, we use uppercase calligraphic for matrices, uppercase italics for sets, and lowercase (bold) italics for (vector) data/variables. For formatting of operators, we use roman and bold when vectorized.
\section{Alternative Solutions}
\label{sec:background:aos}
We define `alternative solutions' as a term that relates to the generation of several solutions to an optimization problem. There are many contexts where alternative solutions have been treated under differing names including `alternative optimal solutions' by \citet{ paris_economic_foundations_of_symmetric_programming} and \citet{hpwilliams_model_building}, `multiple optimal solutions' also by \citet{paris_mult_opt_lp}, `modeling to generate alternatives' by \citet{brill_1982_alternatives_HSJ_hop_skip_jump, brill_1990_mga_tutorial}, `complete local minimizer (CLM)' sets by \citet{clm_sample_path_optimmization_robinson}, `set of all optimal solutions' by \citet{Convex_Analysis_Rockafellar}, and the `set of minimizing points' by \citet{Convex_Analysis_and_Optimization_Bertsekas_Nedić_Ozdaglar_2013}. The different names correspond to research communities that do not seem to interact. Rockafellar and Bertsekas discuss alternative solutions in existence arguments for solution sets. Robinson treats alternative solutions in sample path optimization for simulation for a discussion about set compactness and connectedness. Paris, Williams, and Brill each treat alternative solutions in discussions for their generation. Even in the generation discussions, the different terms date back to different disciplines. The Management Sciences literature starts with \citet{brill_1979_public_planning_alternatives} and the Agricultural Economics literature starts with \citet{paris_mult_opt_lp}. While both areas continued to develop new generation methods \citep[e.g.][]{brill_1982_alternatives_HSJ_hop_skip_jump, paris_multiple_solutions_quadratic} and debate use cases for alternative solutions \citep[e.g.][]{brill_1990_mga_tutorial, paris_to_mccarl_reply}, there does not appear to be a standard notation or representation for alternative solutions either in the early literature or more recent application-specific reviews \citep[e.g.][]{Lau_2024_mga_alternatives_energy_applications_review}.

Since we are concerned both about theory and generation, our presentation divides naturally into these two parts. In Section \ref{sec:background:aos:theory}, we describe a mathematical framework for alternative optimal and near-optimal solutions based on sublevel sets. We use sublevel sets to describe the differences between alternative solutions to the \nameBM\ problem and alternative solutions to the \nameEF\ problem in Section \ref{sec:ext_benders:impact_of_approx}. In Section \ref{sec:background:aos:generation}, we review previous research and available software to generate alternative solutions in a range of problem types.
We leverage these methods to generate alternative solutions for the first-stage variables in Section \ref{sec:ext_benders:master_problem} and for the second-stage variables in Section \ref{sec:ext_benders:sub_problem}.  
\subsection{Alternative Solutions Theory}
\label{sec:background:aos:theory}
 We adapt a mathematical framework to describe alternative solutions for both exact optimal solutions and near optimal solutions from the existence arguments of Bertsekas and Robinson.
Consider the following optimization problem:
\begin{equation*}
    z^* = \min_{\vecx\in \setX} \funcf(\vecx),
\end{equation*}
where $\funcf$ is an objective function defined over feasible domain $\setX$.
The set of \textit{exact alternative optimal solutions} (\aosExactSetName) is the level set:
\begin{equation*}
    \begin{aligned}
        \levelset(\funcf,\setX, z^*) &= \{\vecx\in \setX\ |\ \funcf(x) = z^*\} \\
                            &= \{\vecx\in \setX\ |\ \funcf(x) \leq z^*\}.
    \end{aligned}
\end{equation*}
These are equivalent definitions for this level set since $\{\vecx\in \setX\ |\ \funcf(\vecx) < z^*\} = \varnothing$, but the second definition is the sublevel set with level $z^*$. Sublevel sets are convex for quasi-convex functions, which includes convex functions \citep[see e.g.][]{Convex_Analysis_and_Optimization_Bertsekas_Nedić_Ozdaglar_2013}. 
Further, we consider a general level value of $\tau$ rather than $z^*$:
\begin{equation*}
    \levelset(\funcf,\setX,\tau) = \{\vecx\in \setX\ |\ \funcf(\vecx) \leq \tau\} .
\end{equation*}
When the level value $\tau > z^*$, this is the set of \textit{approximate alternative optimal solutions} (\aosApproxSetName). We can use this for an absolute or relative tolerance from $z^*$ by choosing $\tau = z^* + \epsilon$ or $\tau = (1+\alpha)z^*$ respectively.
\subsection{Generation of Alternative Solutions}
\label{sec:background:aos:generation}
Previous research has developed methods to extract alternative solutions from both Linear Programs (LPs) and Mixed-Integer Programs (MIPs) for both \aosApproxSetName\ and \aosExactSetName\ sets. We draw a distinction between black-box and white-box generation techniques. The black-box techniques are alternative solution generation methods where the implementation details are unknown.
White-box techniques for LPs include iterative MIP methods \citep{lee2000recursive} and simultaneous discovery methods \citep{PEDERSEN2021121294}. Approaches for MIPs include No-Good Cuts methods for 0-1 problems \citep{balas_no_good_cuts} and heuristic methods like keeping track of the N best incumbents in branch-and-bound \citep{eckstein2015pebbl}.
For black-box methods, both Gurobi and CPLEX provide ways of generating alternative solutions as part of their MIP solvers in solution pool structures. Though the solvers provide guarantees, both note challenges when generating alternative solutions for MIPs that mix continuous and discrete variables \citep{gurobi, cplex}.
Modeling languages like AIMMS, GurobiPy, and Pyomo can generate and represent alternative solutions. AIMMS can use CPLEX's Solution Pool \citep{aimms}, and GurobiPy uses Gurobi's Solution Pool \citep{gurobi}. Pyomo \citep{pyomo} can use
Gurobi's solution pool, and it includes solver-agnostic
methods for generating alternative solutions with non-commercial solvers like GLPK and HiGHS \citep{cimor_aos_report}.
In our experiments in Sections \ref{sec:farmer} and \ref{sec:israeli_wood}, we use the Pyomo generation methods for what Algorithm \ref{alg:aos_benders} calls an AOSKernel. For linear programs, we use the \verb|enumerate_linear_solutions| method (hereafter \nameEnumerateLinear), which implements a version of the vertex-enumeration strategy from \citet{lee2000recursive}. For 0-1 problems, we use the \verb|enumerate_binary_solutions| method (hereafter \nameEnumerateBinary), which implements a version of No-Good Cuts method from \citet{balas_no_good_cuts}. Both methods can exhaustively enumerate when the sublevel set is compact, either in terms of vertices for LPs or all points for 0-1 problems.  Such enumeration forms a core part of our analysis of exemplars. Note that alternative solution generation methods can differ in the order they enumerate or report solutions. We standardize on using an optimal search mode which generates the next closest vertex or point to optimal by objective value for LPs and 0-1 problems, respectively.
\section{Benders Decomposition Review}
\label{sec:background:benders}
We review several elements of \nameBenders\ in some detail below to establish context for the proofs of our core results. More detailed treatments of \nameBenders\  and its variants are available elsewhere in the literature \citep[e.g.][]{birge2011introduction, conforti2014integer}.
We consider problems of the form:
\begin{equation*}
    \begin{aligned}
    \EF:\labelspace&\min_{\vecx,\vecy} \funcg(\vecx) + \vecq^T \vecy \\
    &(\vecx, \vecy) \in \setGamma .
    \end{aligned}
\end{equation*}
We call this the \nameEF\  problem ($\EF$), where $\funcg$ is the first-stage value function, $\vecq \in \mathR^{n_2}$, and \(\setGamma \defeq \{(\vecx, \vecy) \in \mathR^{n_1}\times \mathR^{n_2}_+\ \ |\ \vecx \in \setX, \ \matrixW\vecy + \matrixT\vecx = \vech\}\). We call $\vecx$ the first-stage variables and $\vecy$ the second-stage variables. 
We make several assumptions about problem structure. Let \(\funcY(\vecx) \defeq \{\vecy \in \mathR^{n_2}_+\ \ |\ \vecx \in \setX, \ \matrixW\vecy + \matrixT\vecx = \vech\}\). We make the relatively complete recourse assumption: $\forall \barvecx \in \setX\ \exists \barvecy \in \funcY(\barvecx) \ s.t. \  \vecq^T\barvecy < \infty$. We make the dual non-emptiness assumption: $\{\vecpi \in \mathR^{n_3}\ \ |\ \vecpi^T \matrixW \leq \vecq\} \neq \varnothing$ where $n_3 = \dim (\vech)$. We make the assumption a minimizer exists: $(\vecx^*,\vecy^*) \in \setGamma$. We make the finite solution assumption: $z^* \defeq \funcg(\vecx^*) + \vecq^T \vecy^* > -\infty$. The finite solution assumption implies that $\funcg(\vecx) + \vecq^T \vecy$ is bounded below over $(\vecx, \vecy) \in \setGamma$, and can be established by computing $z^*$ or bounding for some $\underline{z} \in \mathR$ as $z^* \geq \underline{z}$ . 

We define a projection operator onto the subspace of the first $n_1$ (or first-stage) variables:
\begin{equation}
    \label{eqn:proj_x_def}
    \projX(\vecp) \defeq \matrixM \vecp, \quad \matrixM \in \mathR^{n_1 \times \dim(\vecp)}, \quad \matrixM_{ij} = \left\{\begin{matrix}
                      1 \quad i=j \\
                      0 \quad i\neq j
                    \end{matrix} \right .
\end{equation}
We then rewrite the \nameEF\  problem in terms of just the first-stage variables as the projected variable problem ($\PV$):
\begin{equation*}
\begin{aligned}
    \PV:\labelspace&\min_{\vecx} \funcg(\vecx) + \funcQ(\vecx)\\
        &\vecx \in \setX . 
\end{aligned}
\end{equation*}

The value (or recourse) function $\funcQ$ is defined for $\vecx \in \setX$:
\begin{equation*}
\begin{aligned}
    \PVF:\labelspace \funcQ(\vecx) \defeq &\min_{\vecy}\  \vecq^T \vecy \\
    & \vecy \in \funcY(\vecx).
\end{aligned}
\end{equation*}
We call this the \namePrimalValueFunction\  ($\PVF$). 

We use the strong duality of linear programs to give another way of computing $\funcQ$:
\begin{equation}
\begin{aligned}
\label{eqn:benders:q_dual_def}
    \DVF:\labelspace \funcQ(\vecx) = &\max_{\vecpi}\  \vecpi^T(\vech - \matrixT\vecx) \\
    & \vecpi \in \setPi,
\end{aligned}
\end{equation}
where \(\setPi =\{\vecpi \in \mathR^{n_3}\ \ |\ \ \vecpi^T \matrixW \leq \vecq\}\).
We call this the \nameDualValueFunction\  ($\DVF$). 
Note $\setPi$ is independent of the argument $\vecx$, is fixed, and as noted before, non-empty by assumption. We write $\funcQ$ in terms of the vertices of $\setPi$ by leveraging the assumptions made earlier. For $\vecx \in \setX$, $\funcQ(\vecx)$ takes one of three possible states: unbounded above, unbounded below, and finite. The relatively complete recourse assumption rules out unbounded above. The  finite solution assumption rules out unbounded below as $\funcg(\vecx) + \vecq^T\vecy > -\infty, \forall (\vecx, \vecy) \in \setGamma$ implies $\funcQ(\vecx) >-\infty$. We are left with the finite case that requires $\setPi$ to be non-empty, since duality requires $\forall \vecx \in \setX,\ \exists \bar{\vecpi} \in \setPi \ s.t.\ \funcQ(\vecx) = \bar{\vecpi}^T(\vech - \matrixT\vecx)$.
We use the Fundamental Theorem of Linear Programming \citep[Prop 3.4.2]{Convex_Analysis_and_Optimization_Bertsekas_Nedić_Ozdaglar_2013} to write $\funcQ$ in terms of the dual vertices:
\begin{equation*}
    \funcQ(\vecx) = \max_{\vecpi \in \setvpi}\  \vecpi^T(\vech - \matrixT\vecx),
\end{equation*}
where $V(S)$ represents the list of vertices of set $S$. This allows us to write an equivalent problem to the \namePV\  problem:
\begin{equation*}
\begin{aligned}
    \EV:\labelspace &\min_{\vecx,\theta}\  \funcg(\vecx) + \theta \\
    & \theta \geq \vecpi^T(\vech - \matrixT \vecx) \quad \forall \vecpi \in \setvpi \\
    & \vecx \in \setX \subseteq \mathR^{n_1}, \theta \in \mathR .
\end{aligned}
\end{equation*}  
We call this the \nameEV\   ($\EV$) problem. For an optimal solution $\vecp^{(2)} = (\barvecx^{(2)}, \theta^{(2)})$ to the \nameEV\ problem, the projection $\projX(\vecp^{(2)}) = \barvecx^{(2)}$ is an optimal solution to the \namePV\  problem. The converse is true: for  and optimal solution $\barvecx^{(3)}$ to the \namePV, $(\barvecx^{(3)}, \barvecy^{(3)})$ is an optimal solution to the \nameEV\  problem, where $\barvecy$ exists by the relatively complete recourse assumption and $\vecq^T \barvecy^{(3)} = \funcQ(\barvecx^{(3)})$. 

Next we approximate $\funcQ$ by considering only a subset of the dual vertices, $\hatsetV \subseteq \setvpi$:
\begin{equation*}
    \funchatQV(x) \defeq \max_{\vecpi \in \hatsetV}\  \vecpi^T(\vech - \matrixT\vecx).
\end{equation*}
Since the maximum function is monotonically increasing, we know that:
\begin{equation}
    \label{eqn:benders:vertex_lower_bound}
\begin{aligned}
    \funcQ_{\setvpi}(\vecx) = \funcQ(\vecx) &\geq \funchatQV(\vecx) \quad \forall \ \hatsetV \subseteq \setvpi, \forall \vecx \in \setX.
\end{aligned}
\end{equation}
We now define a version of the \nameEV\   problem relying on $\hatsetV$ instead of $\setvpi$, where $\hatsetV \subseteq \setvpi$:
\begin{subequations}
\label{eqn:benders:master_problem}
\begin{align}
    \BMhatV:\labelspace &\min_{\vecx,\theta}\  \funcg(\vecx) + \theta \\
    \label{eqn:benders:master_problem_theta_bounds}
    & \theta \geq \vecpi^T(\vech - \matrixT \vecx) \quad \forall \vecpi \in \hatsetV \\
    & \vecx \in \setX \subseteq \mathR^{n_1}, \theta \in \mathR.
\end{align}  
\end{subequations}
We call this the \nameBM\  problem for $\hatsetV$ ($\BMhatV$). 
By construction, the \nameBM\  problem gives the same optimal first-stage solution(s) and objective value when $\hatsetV = \setvpi$. We include a basic version of the single-cut Benders Algorithm as Algorithm \ref{alg:classic_benders}. 
\begin{algorithm}
\caption{Benders Decomposition Algorithm}\label{alg:classic_benders}
\begin{algorithmic}
\Procedure{Benders}{$tol, \ iterLimit, \ \BM, \ \funcQ$} 
\State $\hatsetV \gets \varnothing$
\While{$|\hatsetV| \leq iterLimit$}
\State $Solve(\BMhatV)$, $\barvecx \gets \vecx^*$, $\theta \gets \theta^*$, $z \gets \funcg(\vecx^*) + \theta^*$
\State $\mbox{Solve}(\funcQ(\barvecx))$, $\vecpi \gets \vecpi^*$, $\hattheta \gets \funcQ(\barvecx)$
\If{$\|\theta - \hattheta\| > tol$}
\State $\hatsetV = \hatsetV_{t-1} \cup  \{\vecpi\}$
\Else
\State \textbf{break}
\EndIf
\EndWhile
\State \Return $(\barvecx, \theta, z, \hatsetV)$
\EndProcedure
\end{algorithmic}
\end{algorithm}

\section{Extending Benders to Yield Alternative Solutions}
\label{sec:ext_benders}
While the ability to automatically generate alternative solutions for LPs and MIPs directly from solvers or modeling languages is useful, this capability does not naturally exist for problems solved by \nameBenders. When we use a problem decomposition, we also split up and approximate core structural information that enabled previous tools to do automatic generation of alternative solutions (Section \ref{sec:ext_benders:impact_of_approx}). To address this, we present methods for treating each of the first-stage and the second-stage alternative solutions (Sections \ref{sec:ext_benders:master_problem} and \ref{sec:ext_benders:sub_problem} respectively), resulting in a combined process capable of treating first-stage, second-stage, and \nameEF\ alternative solutions. The overall process becomes: prove a problem meets (or modify to enforce) the assumptions from Section \ref{sec:background:benders}, divide the problem into first-stage and second-stage components to form the \nameBM\ problem and subproblem, apply Algorithm \ref{alg:aos_benders} for tolerances of interest to get first-stage alternative solutions, and then (if needed) apply the LP alternative solution techniques of Section \ref{sec:ext_benders:sub_problem}. We then analyze the resulting alternative solutions for insights into our optimization problems as seen on examples in Sections \ref{sec:farmer} and \ref{sec:israeli_wood}.
\subsection{Impact of Approximation in Benders Decomposition}
\label{sec:ext_benders:impact_of_approx}
We know choice of $\hatsetV \subseteq \setvpi$ can change the objective value and minimizers in the \nameBM\ problem. As a result, we need to consider the impact this approximation has on the \aosExactSetName\ and \aosApproxSetName\ sets, which we do by comparing the sublevel sets. We define the following sublevel sets for the \nameBM\ problem and the \nameEF\  problems:

\begin{equation}
\begin{aligned}
    \label{eqn:extending_benders:master_problem_level_set_def}
    \levelBM(\tau) &= \levelset(\funcG_{\theta}, \epigraph(\funchatQV),\tau) 
\end{aligned}
\end{equation}
\begin{equation}
\begin{aligned}
    \label{eqn:extending_benders:extensive_form_level_set_def}
    \levelEF(\tau) &= \levelset(\funcG_q, \setGamma, \tau),
\end{aligned}
\end{equation}
where $\funcG_{\theta}(\vecx, \theta) = \funcg(\vecx)+\theta$, $\funcG_{q}(\vecx, \vecy) = \funcg(\vecx)+\vecq^T\vecy$, and $\epigraph(\funchatQV)$ is defined relative to $\epigraph(\funcQ)$ as:
\begin{equation}
    \label{eqn:q_epigraph}
        \epigraph (\funcQ) = \{(\vecx,\theta) \in \setX\times \mathR\ \vert\  \theta \geq \vecpi^T(\vech -\matrixT\vecx), \forall \vecpi \in \setvpi\}
\end{equation}
\begin{equation}
    \label{eqn:q_hat_epigraph}
        \epigraph (\funchatQV) = \{(\vecx,\theta) \in \setX\times \mathR \ \vert\  \theta \geq \vecpi^T(\vech -\matrixT\vecx), \forall \vecpi \in \hatsetV\}
\end{equation}
The following theorem gives us the guarantee that any first-stage alternative solutions for the \nameEF\  problem will be alternative solutions in the \nameBM\  problem:
\begin{theorem}
\label{thm:proj_ef_subseteq_proj_benders}
    $\projX(\levelEF(\tau)) \subseteq \projX(\levelBM(\tau)), \ \forall \hatsetV \subseteq \setvpi, \forall \tau \in \mathR$.
\end{theorem}
\begin{proof}{Proof}
    This follows directly by the nature of the projection operator of \equationRef{eqn:proj_x_def} and Lemmas 3 and 4 from the Proof Appendix. 
\end{proof}
Since this result is defined over the first-stage variables, it is abstracting away the details of the second-stage variables for the \nameEF\ problem. If $\projX(\levelEF(\tau)) \neq \varnothing$ for some fixed $\tau \in \mathR$, we know there are feasible second-stage points in the sublevel set (e.g. $\barvecx \in \projX(\levelEF(\tau)) \rightarrow |(\{\barvecx\}\times \funcY(\barvecx))\cap \levelEF(\tau)| \geq 1$).

The core insight for the sublevel set properties comes from analyzing $\epigraph (\funcQ)$ and $\epigraph (\funchatQV)$, which gives the following result:
\begin{theorem}
\label{thm:epigraph_containment}
    \(\epigraph (\funcQ) \subseteq \epigraph (\funchatQV),\forall \  \hatsetV \subseteq \setvpi\)
\end{theorem}
\begin{proof}{Proof}
The inclusion holds trivially if $\epigraph(\funcQ) = \varnothing$. In the case $\epigraph(\funcQ) \neq \varnothing$, we consider $(\barvecx,\bartheta) \in \epigraph(\funcQ)$. We have $(\barvecx,\bartheta) \in \setX \times \mathR$ by definition of $\epigraph(\funcQ)$. All that remains is showing that $\bartheta \geq \vecpi^T(\vech -\matrixT \barvecx), \forall \vecpi \in \hatsetV \subseteq \setvpi$ holds and this follows directly from $\bartheta \geq \vecpi^T(\vech -\matrixT \barvecx), \forall \vecpi \in \setvpi$ as part of the definition of $\epigraph(\funcQ)$. So $(\barvecx,\bartheta) \in \epigraph(\funchatQV)$ giving inclusion in the non-empty case.
\end{proof}
We do know that the converse results are not true in general:
\begin{remark}
$\projX(\levelBM(\tau)) \subseteq \projX(\levelEF(\tau))$ does not hold in general.
\end{remark}
\begin{remark}
    $\epigraph(\funchatQV) \subseteq \epigraph (\funcQ)$ does not hold in general.
\end{remark}
The core intuition for all of the results can be seen in the following example where $\funcQ(\vecx) = |\vecx|$:
\begin{figure}[H]

\subcaptionbox{$\epigraph(\hatQVVariant{\hatV_1}), \ \hatQVVariant{\hatV_1}(x) = x$\label{fig:q_abs_x:v_1}}
{\includegraphics[width=0.3\textwidth]{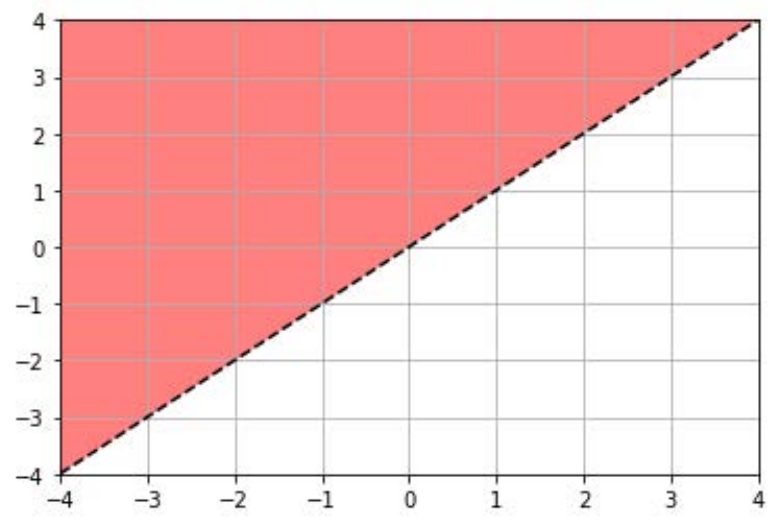}}
\hfill\subcaptionbox{$\epigraph(\hatQVVariant{\hatV_2}), \ \hatQVVariant{\hatV_2}(x) = -x$\label{fig:q_abs_x:v_2}}
{\includegraphics[width=0.3\textwidth]{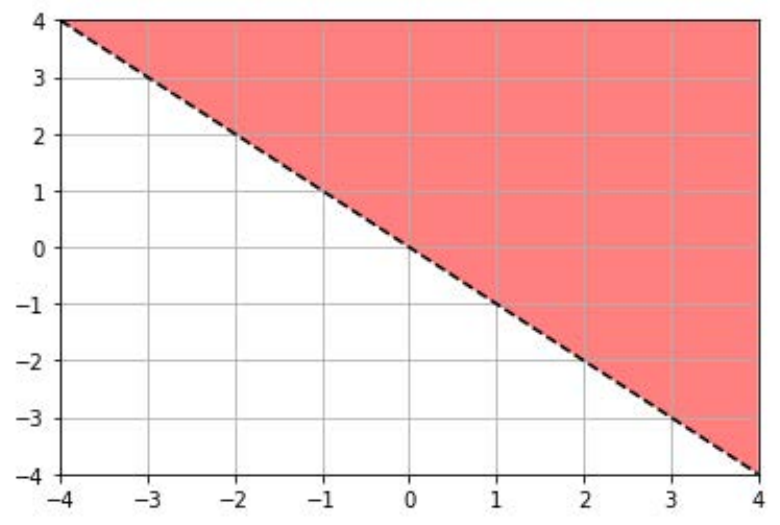}}
\hfill\subcaptionbox{$\epigraph(\hatQVVariant{\hatV_1}), \ \hatQVVariant{\hatV_1}(x) = |x|$\label{fig:q_abs_x:v_3}}
{\includegraphics[width=0.3\textwidth]{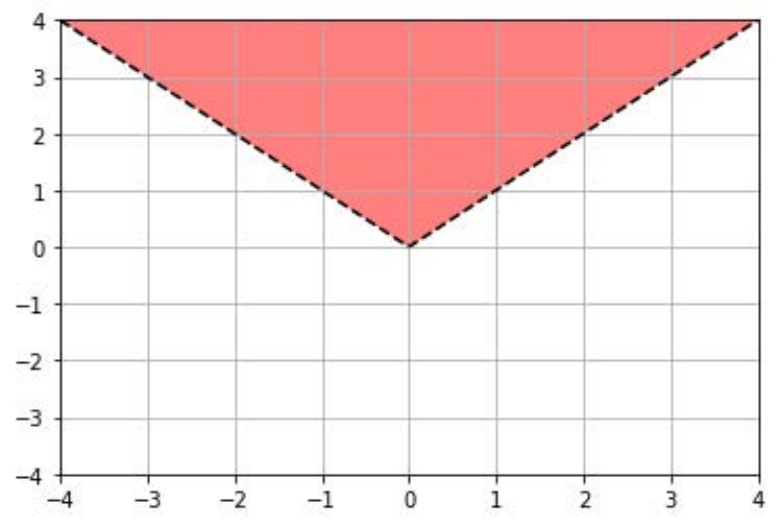}}
\caption{
Examples of $\epigraph(\funchatQV)$ for several values of $\hatsetV$. $\hatsetV_1 = \{1\}$, $\hatsetV_2 = \{-1\}$, $\hatsetV_3 = \{-1,1\}$}
\label{fig:q_abs_x:full_figure}
\end{figure}
We see the epigraph containment result of Theorem \ref{thm:epigraph_containment}: the first two plots contain the third. The converse of Theorem \ref{thm:epigraph_containment} does not hold for this example. The figure also gives us an intuition about Theorem \ref{thm:proj_ef_subseteq_proj_benders} for two reasons. First, the main difference between the the \nameEF\  and \nameBM\ problems is the difference between the exact and approximate value functions. Second, we see the geometric connection between the epigraph and sublevel sets of a function.

\subsection{Extracting Alternative Solutions for First-Stage Variables}
\label{sec:ext_benders:master_problem}
The \nameEF\  and \nameBM\  problems share the first-stage variables. As a consequence of this, we break the generation of alternative solutions to the \nameEF\  problem into two steps. First, we find alternative solutions over the first-stage variables. Second, given a solution for the first-stage variables we generate alternative solutions for the second-stage variables, which we defer to Section \ref{sec:ext_benders:sub_problem}. We note that there are cases where all we need are first-stage decisions (e.g. two-stage stochastic programming for Capacity Expansion); in such cases, being able to generate first-stage decisions separately is a major benefit of AOS-Benders method. We present a three-step method in Algorithm \ref{alg:aos_benders} to achieve this first-stage alternative solution generation.

\begin{algorithm}
\caption{AOS-Benders Algorithm}\label{alg:aos_benders}
\begin{algorithmic}
\Procedure{Benders}{$BendersTol,\ OptTol,\ iterLimit, \ solLimit, \ BM, \ \funcQ, \ \funcg, \ \AOSKernel$}
\State $(\vecx^*, \theta, z^*, \hatsetV) \gets \mbox{Benders}(BendersTol, \ iterLimit,\ BM,\ \funcQ)$ \Comment{Step 1: \textit{Benders Solve}}
\State $\tau \gets z^* + OptTol$
\State $\funcS_{BM} \gets \AOSKernel(\BMhatV, solLimit, \tau)$ \Comment{Step 2: \textit{Generate Candidates}}
\State $\funcS_{proj(EF)} \gets \varnothing$
\For {$(\barvecx, \bar{\theta}) \in \funcS_{BM} $} \Comment{Step 3: \textit{Certify Candidates}}
\State $\mbox{Solve}(\funcQ(\barvecx)), \ \theta \gets \funcQ(\barvecx)$ 
\If{$\funcg(\barvecx) + \theta \leq \tau$}
\State $\funcS_{proj(EF)} \gets \funcS_{proj(EF)} \cup \{\barvecx\}$
\EndIf
\EndFor
\State \Return $\funcS_{proj(EF)}$
\EndProcedure
\end{algorithmic}
\end{algorithm}

There are three steps in AOS-Benders. The first step is \textit{Benders Solve}, which solves the \nameBM\ problem to optimality and returns the optimal objective value, $z^*$, and terminating cuts, $\hatsetV$. For example, we can apply Algorithm \ref{alg:classic_benders}. In step two, \textit{Generate Candidates}, we generate alternative solutions to $\levelBM(\tau)$. We do this by calling a function, $AOSKernel$, to generate alternative solutions for the Benders master problem for $\hat{V}$. For example, $AOSKernel$ can be one of the methods discussed in Section \ref{sec:background:aos:generation} that is suitable for the Benders master problem. The final step is to \textit{Certify Candidates}, which filters the points generated in Step 2 to keep only solutions that are in $\projX(\levelEF(\tau))$. For each $(\barvecx, \bartheta) \in \funcS_{BM}$, we test if 
$\funcg(\barvecx) + \funcQ(\barvecx) \leq \tau$, which is sufficient to guarantee that $\barvecx \in \projX(\levelEF(\tau))$ from the definition \equationRef{eqn:extending_benders:extensive_form_level_set_def}.
\subsection{Extracting Alternative Solutions for Second-Stage Variables}
\label{sec:ext_benders:sub_problem}
Alternative solutions for second-stage variables rely on having the first-stage variables fixed to a specific value. We assume that we have some fixed $\barvecx \in \setX$. We discussed in Section \ref{sec:ext_benders:impact_of_approx} how the approximation in the \nameBM\  problem impacts generating alternative solutions to the \nameEF\  problem, but this impact was localized to the first-stage variables and therefore the \namePV\  problem. As a result, nothing in the value function $\funcQ$ an approximation, which means that the generation of alternative solutions for the second-stage variables can ignore the details of the \nameBM\  problem. 

Here we are concerned with taking a previously-known point, $\barvecx \in \setX$, from the \namePV\  problem and generating alternative solutions for the second-stage variables. This will also let us generate alternative solutions to the overall \nameEF\ problem. The value function, $\funcQ$, is an LP so we can apply any of the alternative solutions generation methods for LPs described in Section \ref{sec:background:aos:generation} like \nameEnumerateLinear. 

To generate alternative solutions to the \nameEF\ problem, we need to address the difference between the $\projX(\levelEF(\tau))$  and $\levelEF(\tau)$  sublevel sets. We recall the definitions as:
\begin{equation*}
    \levelEF(\tau) = \levelset(\funcG_q, \setGamma, \tau) \subseteq \mathR^{n_1 + n_2}.
\end{equation*}
Note that the definition of $\levelEF(\tau)$ makes no use of the fact that we have a fixed $\barvecx \in \setX$. We use this to define the sublevel set of the remaining $\vecy$ options for a fixed $\barvecx \in \setX$:
\begin{equation*}
    \levelset^{P}(w; \barvecx) = \levelset(\vecq^T\vecy, \funcY(\barvecx), w).
\end{equation*}
We also then need to take into account the combined optimality tolerance for \nameEF\ alternative solutions by setting the second-stage optimality tolerance in terms of both the overall tolerance, $\tau$, and the impact of the first-stage decision, $\barvecx \in \setX$. We see how this works in the following theorem to reconstruct combined alternative solutions for the \nameEF:
\begin{theorem}
    Suppose $\tau$ given and $\barvecx \in \projX(\levelEF(\tau))$. Let $\hatvecy \in \levelset^{P}(\tau - \funcg(\barvecx)); \barvecx)$, then $(\barvecx, \hatvecy) \in \levelEF(\tau)$.
\end{theorem}
\begin{proof}{Proof}
    Show that $(\barvecx, \hatvecy) \in \setGamma$ and $\funcg(\barvecx) + q^T\hatvecy \leq \tau$ to recover the definition of $\levelEF(\tau)$. \\
    $\barvecx \in \projX(\levelEF(\tau))$ gives $\barvecx \in X$ by definition. Since $\hatvecy \in \levelset^{P}(\tau - \funcg(\barvecx))$ gives $ \hatvecy \in Y(\barvecx)$, then holds $(\barvecx, \hatvecy) \in \setGamma$ by construction. 
    Since $\hatvecy \in \levelset^{P}(\tau - g(\barvecx); \barvecx)$, we know by construction $q^T\hatvecy \leq \tau - g(\barvecx)$. Rearranging gets $\funcg(\barvecx) + q^T \hatvecy \leq \tau$.
\end{proof}
This makes generating extensive-form alternative solutions for fixed first-stage decisions, $\barvecx \in \setX$, a matter of generating linear programming alternative solutions with the second-stage optimality tolerance $w = \tau - \funcg(\barvecx)$. As a result, generating extensive-form alternative solutions is an optional post-processing step to making first-stage alternative solutions with Algorithm \ref{alg:aos_benders}.

\section{Application: Farmer's Problem}
\label{sec:farmer}

We now illustrate how AOS-Benders generates alternative solutions for stochastic programming problems.  We consider a farmer planning problem from \citet{birge2011introduction}  where the farmer has a set of 3 crops (wheat, corn, and sugar beets) and cattle to feed. In the first-stage, the farmer controls how many acres of land of each crop are planted. In the second-stage, the farmer either buys or sells crops to meet a cattle feed target and to minimize financial loss.
\subsection{Model}
\label{sec:farmer:models}
The scenario-based stochastic version with variable crop yields is:
    \begin{equation*}
        \label{eqn:farmer:stochastic}
    \begin{aligned}
        \min_{\vecx} & (150x_1 + 230x_2 + 260x_3) \\
    & \quad+\sum_{\omega\in\setOmega}p^{(\omega)}\left[(238y_1^{(\omega)}+210y_2^{(\omega)}) - (170w_1^{(\omega)} + 150w_2^{(\omega)} + 36w_3^{(\omega)} + 10 w_4^{(\omega)})\right] \\
    & x_1 + x_2 + x_3 \leq 500 \\
    & \xi_{1}^{(\omega)} x_1 + y_{1}^{(\omega)} - w_{1}^{(\omega)} \geq 200\;\forall\;\omega\in\setOmega\\
    & \xi_{2}^{(\omega)} x_2 + y_{2}^{(\omega)} - w_{2}^{(\omega)} \geq 240\;\forall\;\omega\in\setOmega\\
    & \xi_{3}^{(\omega)} x_3 - w_{3}^{(\omega)} - w_{4}^{(\omega)} \geq 0\;\forall\;\omega\in\setOmega\\
    & w_{3}^{(\omega)}  \leq 6000 \\
    & \vecx,\vecy^{(\omega)},\vecw^{(\omega)}\geq \veczero\;\forall\;\omega\in\setOmega
        \end{aligned}
    \end{equation*}
Here $\vecxi^{(\omega)}$ is a vector denoting the yields-per-acre planted under scenario $\omega$. There are three crops: wheat (1), corn (2), and beets (3 \& 4). $\vecx$ is the number of acres of each crop planted. $\vecy$ and $\vecw$ are the number of tons of crops purchased and sold respectively. $\vecw$ has two values for sugar beets to reflect the two sale points. The following equations are used to decompose this model into the form used by \nameBenders:
\begin{equation*}
\begin{aligned}
    \funcg(\vecx) &= 150x_1 + 230x_2 + 260x_3 \\
    \setX &= \{\vecx \in \mathR^3_+ \ \ |\ \ x_1 + x_2 + x_3 \leq 500\} \\
    \funcQ(\vecx) &= \sum_{\omega \in \setOmega} p^{(\omega)} \funcQ_k(\vecx, \vecxi^{(\omega)})
\end{aligned}
\end{equation*}
\begin{equation*}
    \begin{aligned}
    \label{eqn:farmer:subproblems}
        \funcQ_k(\vecx, \vecxi) = &\min_{\vecy, \vecw, \vecu} (238y_1+210y_2) - (170w_1 + 150w_2 + 36w_3 + 10 w_4) \\
        & \xi_1 x_1 + y_1 - w_1 \geq 200\\
        & \xi_2 x_2 + y_2 - w_2 \geq 240\\
        & \xi_3 x_3 - w_3 - w_4 \geq 0\\
        & w_3 \leq 6000 \\
        & \vecy,\vecw, \vecu \geq \veczero
    \end{aligned}
\end{equation*}
Again, slack variables $\vecu$ can be introduced to convert from inequality to equality constraints, but we omit them to simplify our presentation. 

Both the single-scenario and multiple-scenario cases of this problem are linear programs with continuous variables. This means there are convex sublevel sets in both the \nameEF\  and \nameBM\  problems. We also have relatively complete recourse in the value function, because we can always choose $\barvecy = [200 -\xi_1 x_1 ,\ 240 -\xi_2 x_2]^T$ and $\vecw = 0$ for a cost of $[238,\ 210]\barvecy$ for $\vecx \in \setX$. We note that we have at least one dual vertex because $\funcQ(\veczero, \vecxi) = [238,\  210][200, \ 240]^T = 98000$ for all $\vecxi \in \mathR^3$. We have $z^*$ bounded below because $\funcQ(\vecx, \vecxi) \geq -500 \max\{170 \xi_1, 150 \xi_2, 36 \xi_3\}$, which corresponds to growing only the most profitable crop and selling it all to the market. So long as $\vecxi \in \mathR^3$, we have $z^* \in \mathR$ satisfying the finite solution assumption. This suffices to demonstrate that both the deterministic and stochastic Farmer's problem fit into our framework and assumptions.

\subsection{Meaning of Alternative Solutions}
\label{sec:farmer:meaning_of_alts}

There are two key things to note about the meaning of these alternative solutions. First, we can generate alternative solutions in the $\setX$ (or \namePV) space. As this is a staged stochastic problem, this means that we can generate alternative solutions for first-stage variables without generating corresponding second-stage variables. This significantly reduces the complexity of generating alternative solutions for first- and second-stage variables. Even in this simple problem with $N$ scenarios, this reduces the size of the variable space to explore first-stage alternative solutions from $\mathR^{3+5N}_+$ to $\mathR^{3}_+$. 

Second this is a continuous problem, so we may have an infinite number of points in any \aosApproxSetName\ set even if the \aosExactSetName\ set has a single point. This means that we need to consider techniques that emphasize discovery of the ``interesting" or ``meaningful" alternative solutions. What ``interesting'' means can be problem-specific. In some cases, a structured vertex representation may suffice as with the \nameEnumerateLinear\ method we use here. In others, solution diversification strategies may be needed \citep[e.g.][]{danna2007_multiple_solutions_for_mixed_integer_programs_mips_ILOG_CPLEX,Petit2019_alternative_solutions_as_solution_engineering_review_ijoc}.
\subsection{Extracting Alternative Solutions from the Farmer's Problem}
\label{sec:farmer:example}
\subsubsection{Single-Scenario Problem} 
We tested AOS-Benders on the single-scenario mean-yield Farmer's Problem by choosing $|\setOmega| = 1$ and $\vecxi = [2.5, 3, 20]^T$. After 9 iterations and 8 cuts, the Benders Algorithm converges to the known optimal point $z^* = -118,600$, $\vecx^* = [120, 80, 300]^T$, $\vecy^* = [0,0]^T$, $\vecw^* = [100, 0, 6000]$. 
We use the \nameEnumerateLinear\ method from Pyomo 6.8.1 to generate alternative solutions for both the \nameBM\  problem (i.e. the Algorithm \ref{alg:aos_benders} AOSKernel) and subproblem. This method enumerates up to $K$ vertices subject to an optimality tolerance according to a search mode. We use the `optimal' search mode, which orders vertices on the basis of the original objective, and we use an absolute optimality tolerance relative to the optimal objective value. In the limit, this reduces to exhaustive discovery of the vertices of the constraint space. As a result, when less than $K$ solutions are returned, we have exhaustively enumerated the feasible vertices. 

First, we enumerated exact optimal solutions ($0\%$ optimality tolerance) and found that there was only one point, the optimal solution to the \nameBM\  problem for these cuts. By Theorem \ref{thm:proj_ef_subseteq_proj_benders}, we know that all the exact optimal first-stage solutions to the \nameEF\  problem are contained in this set. So we have found the one exact optimal first-stage candidate solution to the \nameEF\  problem and have it certified as a true solution by the termination condition for Algorithm \ref{alg:classic_benders}.  Next we enumerate for the subproblem given $\vecx^*$: we find only one solution, which is the known recourse decision as an exact optimal solution. This matches expectations: the farmer is likely to take a unique cost-optimal recourse. 

Second, we enumerated approximate optimal solutions given a $1\%$ optimality tolerance; given $z^* = -118,600$, this means means we consider solutions with objective values in $[-118,600, -117,414]$ (or $\levelBM(\tau)$, $\tau = -117,414$). We set $K=10$ and generated 6 points ($\vecx^*$ and 5 new points), which are the vertices of $\projX(\levelBM(\tau))$. All 5 new solutions had objectives of $-117,414$ indicating that they are exactly along the $1\%$ optimality boundary. All 5 of these new vertices in $\projX(\levelBM(\tau))$ then passed the certification step in Algorithm \ref{alg:aos_benders}. We then know by Theorem \ref{thm:proj_ef_subseteq_proj_benders} and compactness of $\setX$ that we have $\projX(\levelEF(-117,414)) = \projX(\levelBM(-117,414))$. This means we have, and can prove we have, a representation of the $1\%$ optimal first-stage decisions as the convex combination of these 6 points.
\subsubsection{Multiple-Scenario Problem}
We tested our generation of alternative solutions on the stochastic farmer's problem with three scenarios. The first scenario was the same as the mean yield used in the single-scenario version (i.e. $\vecxi = [2.5, 3, 20]^T$). The second and third scenarios were yields of 20\% above and below mean across crops.
It took 11 iterations and 10 cuts for \nameBenders\  to converge.  The expected optimal solution of $\vecx^*=[170, 80, 250]$ with $z^* = -108,390$ was discovered. 
We use the same enumeration methods as in the deterministic case with the same exact and then approximate solution exploration approach. When we use \nameEnumerateLinear\ with optimality tolerance $0\%$, we get only one point: $\vecx^*$. By Theorem \ref{thm:proj_ef_subseteq_proj_benders}, we conclude $\projX(\levelEF(z^*)) = \{\vecx^*\}$, meaning we have found the only exact optimal first-stage decision. We then generated alternative solutions for second-stage variables given $\vecx^*$, and we again get only one solution per scenario.

Next, we consider the generation of approximate solutions with two optimality tolerances: $1\%$ and $50\%$. When the optimality tolerance is $1\%$, $\tau = -107,306.1$.  We get 15 possible solutions from \nameEnumerateLinear\ when $K = 50$; these are the 15 vertices for $\levelBM(-107,306.1)$. After we apply the certification step of Algorithm \ref{alg:aos_benders}, only 11 points are in $\projX(\levelEF(-107,306.1))$ for the first-stage decisions. In the $50\%$ optimality case ($\tau = -54,195$) again with $K=50$, we get 43 vertices for $\levelBM(-54,195)$. After the certification step, only 29 are in $\projX(\levelEF(-54,195))$ for the first-stage decisions. In both $1\%$ and $50\%$ optimality tolerance, we discover $\projX(\levelBM(\tau)) \neq \projX(\levelEF(\tau))$ since not all vertices from $\projX(\levelBM(\tau))$ are admitted to $\projX(\levelEF(\tau))$.
\subsection{Discussion}
This example illustrates that we can generate alternative solutions for stochastic problems solved by \nameBenders. We are able to prove the uniqueness of the optimal crop planting strategy even with the problem decomposed in both the deterministic and stochastic cases. We also see that in both the one- and three-scenario cases the majority of the \aosApproxSetName\ \nameBM\ solutions are also first-stage solutions to the corresponding \aosApproxSetName\ \nameEF\ problem. This shows that even approximate alternative solutions discovered using \nameBenders\ can map to alternative solutions to the first-stage \nameEF\  problem. This is likely a function of having sufficient density of cuts near the optimal solutions. Finally, we are able to determine when $\projX(\levelBM(\tau)) = \projX(\levelEF(\tau))$ holds by exhausting the vertices of $\projX(\levelBM(\tau))$.

\section{Application: \IW\  Shortest Path Interdiction}
\label{sec:israeli_wood}
The following example illustrates the use of AOS-Benders on an interdiction problem described by \citet{israeli_wood}.  They present a variant on the shortest path interdiction problem described as ``maximize the shortest $s-t$ path length in a directed network by interdicting arcs" or the Maximizing the Shortest Path (MXSP) problem. They treated this problem with both \nameBenders\  and several advanced \nameBenders\  variants.  
MXSP models two agents, a defender, and an attacker. The defender (or ``network user'') wants to cross from $s$ to $t$. The attacker (or ``interdictor'') wants to make that as high a cost as possible. In this version of the problem, the attacker is worsening the best-case $s-t$ traversal in the network, and the objective value is the cost of that best traversal.

Knowing whether multiple traversals exist that achieve the optimal value has a clear value in the security context, and identifying what (or some of) those optimal traversals are also has clear value in the security context.
We can also move beyond the optimal traversal to consider near optimal traversals. This is especially useful when the instance data is an estimate or noisy. Asking the same two questions about near-optimal traversals has clear value if this attacker-defender problem is used to plan defender reaction. Similarly, the same questions about optimal traversal existence and identification apply to optimal (and near-optimal) attacks as well.
\subsection{Models}
\label{sec:israeli_wood:models}
We adapt the Extensive Form and Bender Decomposition models that \IW\ generated for this problem (specifically their Algorithm 1 and 1-E models). 

\subsubsection{Original Formulation}
\label{sec:israeli_wood:models:original_form}

The MXSP interdiction problem is a max-min problem where the attacker decides how to interdict paths for the defender. The defender minimizes its cost to traverse from $s$ to $t$ on graph $G=(\nodeSet,\arcSet)$. 
The defender traversal route is recorded by $y_k$ where $y_k = 1$ if traversed and $0$ otherwise. This means that a defender path is defined as $\{k \in \arcSet\ \ |\ \ y_k = 1\}$. The attacker interdiction choices are recorded by $x_k$ where $x_k = 1$ if interdicted and $0$ otherwise. 
The source node is $s$ and the sink node is $t$. We use special sets to denote arcs entering and exiting node $i$ as $\funcRS(i)$ and $\funcFS(i)$ respectively. Arcs have two components to their traversal cost, the baseline cost $c_k \in [0, \infty)$ and the added cost if interdicted $d_k \in [1, \infty)$. The attacker has a budget of $m \in [1,\infty)$ available to spend, and each arc has an interdiction cost $r_k \in [1, \infty)$.
Note that we do not explicitly enforce the integrality of $\vecy$. The constraint structure of the inner minimization problem is totally unimodular, which induces integrality \citep[see][Section 4.2]{conforti2014integer}. We assume that our solver returns a vertex solution (e.g. like the simplex method).

The max-min model fits into our \nameEF\ structure in Section \ref{sec:background:benders} by applying the transform $\max_{a \in A} \funcf(a) =\min_{a \in A} -\funcf(a)$. Thus we have:
\begin{equation}
\label{eqn:israeli_wood:extensive_form_our_notation}
\begin{aligned}
    \min_{\vecx \in \setX} \min_{\vecy} &-\sum_{k \in \arcSet} (c_k+x_kd_k)y_k \\
    \sum_{k \in \funcFS(i)} y_k - \sum_{k \in \funcRS(i)} y_k &= \left \{ \begin{matrix}
        1 & i=s \\
        -1 & i =t \\
        0 & \forall i \in \nodeSet\backslash \{s,t\}
    \end{matrix} \right. \\
    y_k &\geq 0 \ \ \forall k \in \arcSet \\
    \setX &= \{\vecx \in \{0,1\}^{|\arcSet|}\ \ |\ \vecr^T \vecx \leq m\}
\end{aligned}
\end{equation}
The first-stage elements of \nameBenders\ are:
\begin{equation*}
\begin{aligned}
    \funcg(\vecx) &\defeq 0, \  \setX \defeq \{\vecx \in \{0,1\}^{|\arcSet|}\ \ |\ \vecr^T \vecx \leq m \}.
\end{aligned}
\end{equation*}
The second-stage elements are:
\begin{subequations}
\label{eqn:iw_q_def}
\begin{align}
\label{eqn:iw_q_def:objective}
\funcQ(\vecx) = &\min_{y} -\sum_{k \in \arcSet} (c_k+x_kd_k)y_k \\
    s.t. \labelspace\sum_{k \in \funcFS(i)} y_k &- \sum_{k \in \funcRS(i)} y_k = \left \{ \begin{matrix}
        1 & i=s \\
        -1 & i =t \\
        0 & \forall i \in \nodeSet\backslash \{s,t\}
    \end{matrix} \right. \\
   & y_k \geq 0 \ \ \forall k \in \arcSet 
\end{align}
\end{subequations}
This interdiction problem takes on special meaning; the $\funcQ$ is the value function of the defender's response to an attack. 

\subsubsection{Benders Decomposition Form}
\label{sec:israeli_wood:models:classic_benders}

The resulting model fits into Section \ref{sec:background:benders} format:
\begin{subequations}
\label{eqn:iw_benders_master_problem_our_format}
\begin{align}
    \min_{\vecx \in \setX, u \in \mathR} \quad &u \\
    \label{eqn:iw_benders_master_problem_our_format:cuts}
     u&\geq \ -\vecc^T \vecy - \vecx^T \matrixD\vecy \ \ \forall \hatvecy \in \hatsetV
\end{align}
\end{subequations}
Here $\matrixD = diag(\vecd)$.  The constraint structure in \equationRef{eqn:iw_benders_master_problem_our_format:cuts} matches the structure of the objective in \equationRef{eqn:iw_q_def:objective}. This leads to an interpretation of $\funcQ$ in \equationRef{eqn:iw_q_def} as the dual definition of the value function, (\ref{eqn:benders:q_dual_def}). Both of these features are interdiction modeling structure addressed in \citet{brown_defending_critical_infrastructure}.

The overall approach to solving the MXSP problem is to generate select candidate paths through the network via the subproblems, and only needing to optimize against those in the master problem. It is the selective generation of candidate paths that motivates using \nameBenders\  for this problem. We only consider the ``short'' paths for the attacker to interdict in the \nameMaster\  problem and the delayed constraint generation approach allows us to only consider those paths ``short" enough to be of interest to the attacker.
We compare this to the \nameMaster\  problem that \IW\  define as:
\begin{equation}
\label{eqn:israeli_wood:benders_master_problem}
\begin{aligned}
    \mbox{Master}(\hatsetV)\mbox{-1a:}\labelspace &\max_{\vecx \in \setX, z\in \mathR} z \\
    z \leq & \ \vecc^T\hat{\vecy} + \vecx^T\matrixD\hatvecy \ \ \forall \hatvecy \in \hatsetY
\end{aligned}  
\end{equation}
We see that the difference between the two \nameMaster\  problems is $u^* = - z^*$. Note that \IW\  call their dual vertices object $\hatsetY$. In this problem, each dual vertex, $\hatvecy$, has meaning as a specific defender $s-t$ path,  making the collection of paths $\hatsetY$. We maintain our format for standard dual vertices, $\hatsetV$ and $\setvpi$, since it is more general.

\subsubsection{Analyzing the \IW\  Model}
\label{sec:israeli_wood:models:analysis}

There are several structural components of this model that relate to our \nameBenders\  methodology described in Section \ref{sec:background:benders}. The Benders Form of \IW\  relies on a specific part of the max-min structure that lets the subproblem of \equationRef{eqn:iw_q_def} serve as the DVF form of the value function. This enables the cuts in \equationRef{eqn:iw_benders_master_problem_our_format} to be defined over the same variables as in the \nameEF. This is a more general principle of interdiction modeling and relates to the difference between ``capacity" and ``cost" interdiction modeling discussed in \citet{brown_defending_critical_infrastructure}.
The subproblem has an LP structure, and it is also clear that optimal vertices are integer valued as a result of the total unimodularity of the network flow constraint structure. Additionally, the subproblem is the well known shortest s-t path problem, and Israeli and Wood leverage specialized algorithms to find shortest paths \citep[i.e.][]{Byers_Waterman_Shortest_Path}. 

The dual non-emptiness assumption corresponds to the existence of a feasible $s-t$ flow. We satisfy relatively complete recourse and finiteness assumptions by noting the optimal cost is bounded above by the fully interdicted s-t path 
and bounded below by $\funcQ(0)$.
Thus, our assumptions are satisfied if $\arcSet$ is finite and $c_k, d_k \in \mathR_+, \forall k \in \arcSet$.
The original problem is not convex because the $\vecx$ are integral, so the \nameMaster\  problem remains nonconvex. \nameBenders\  splits the MILP structure into a binary program and a linear program. This has advantages for the generation of alternative solutions, we can avoid limitations of commercial solvers when generating alternative solutions (see Section \ref{sec:background:aos:generation}).

\subsection{Extracting Alternative Solutions From \IW}
\label{sec:israeli_wood:examples}

We demonstrate our methods on the following MXSP problem. We exhaustively generate possible optimal actions of the attacker and defender. This allows us to examine the overlap in optimal actions for both agents across the number of attacked arcs, which goes beyond a traditional sensitivity analysis. Our examples rely on the following directed graph where we flow from $s$ to $t$:
\begin{figure}[H]
\centering
\includegraphics[width=.8\textwidth]{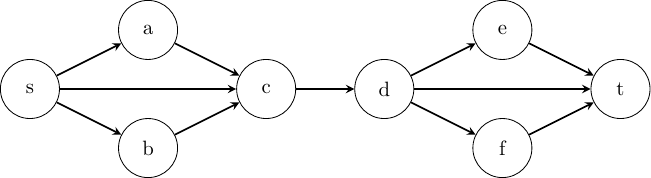}
\caption{Network used for s-t Flow Interdiction. }\label{ig:israeli_wood:examples:exp}

\end{figure}

We set $c_k = 1$ and $d_k = 3$ for all arcs $k \in \arcSet$. All arcs have the same default cost to transit, and interdicted arcs are still passable but cost enough to force avoidance if possible. We allow the attacker to choose $m$ arcs, so $\setX=\{\vecx \in \{0,1\}^{|\arcSet|}\ \ |\ \ \sum_{k \in \arcSet} x_k \leq m\}$. We apply our AOS-Benders process to generate alternative solutions. We use \nameEnumerateBinary\ from Pyomo 6.8.1 to generate alternative solutions for the \nameBM\  problem (i.e. the Algorithm \ref{alg:aos_benders} AOSKernel). We use \nameEnumerateLinear\ to generate alternative solutions for the subproblem. In both cases, we use the optimal search mode with the GLPK solver. We report our results in the max-min sense, using non-negative objective values to match the path cost in the MXSP problem.

\subsubsection{Attack 1 Arc Example}

When the attacker may only attack one arc and the attack impacts are all the same, it makes intuitive sense that the attacker picks arcs that the defender must necessarily traverse to travel from $s$ to $t$ (i.e. chokepoints). In this network, the only chokepoint is $c \rightarrow d$. 
By inspection, we see that this is the unique optimal attack. The attack to interdict $c \rightarrow d$ is the only answer that achieves a defender best traversal cost of 6.
However, the \nameBM\  problem can converge to the optimal cost of 6 with the interdiction choice as $c \rightarrow d$ with different cutpools. The following cuts serve as a certificate of \nameBM\  converging with $\hatsetV \neq \setvpi$ and the terminating cutpool, $\hatsetV$, we encountered is:
\begin{equation*}
\begin{aligned}
    z &\leq 3 + 3(x_{s,c} + x_{c,d} + x_{d,t}) \\
    z &\leq 4 + 3(x_{s,a} + x_{a,c} + x_{c,d} + x_{d,t}) 
\end{aligned}
\end{equation*}

When we run \nameEnumerateBinary\ to exhaustion on \nameBM\ problem: $\BMhatV$, it returns two options with an optimal cost of 6, interdict $c \rightarrow d$ or interdict $d \rightarrow t$. Since we use the optimal mode with a tolerance of zero, we know that these are the only two possible exact optimal solutions for these cuts by Theorem \ref{thm:proj_ef_subseteq_proj_benders}.
The two options make sense for the cuts in the \nameMaster\  problem, but interdicting $d \rightarrow t$ needs to be put through the rest of the AOS-Benders' three-step process. In the certification stage, interdicting $d \rightarrow t$ has a defender cost of 4, and is therefore suboptimal. As a result, we recover only one exact optimal solution for the attacker:interdicting $c \rightarrow d$. We know by Theorem \ref{thm:proj_ef_subseteq_proj_benders} this is the only such solution that achieves an optimal cost of 6. We also generated the alternative solutions for the defender to the optimal attack using \nameEnumerateLinear. The only exact optimal defender path is $s \rightarrow c \rightarrow d \rightarrow t$.

\subsubsection{Attack 2 Arcs Example}
\label{sec:israeli_wood:examples:exp_2}

When the attacker interdicts 2 arcs, our Benders Algorithm terminates optimally by interdicting $c \rightarrow d$ and $s \rightarrow c$ with the following cuts, $\hatsetV$:
\begin{equation*}
\begin{aligned}
    z &\leq 3 + 3(x_{s,c} + x_{c,d} + x_{d,t}) \\
    z &\leq 4 + 3(x_{s,a} + x_{a,c} + x_{c,d} + x_{d,t}) \\
    z &\leq 4 + 3(x_{s,c} + x_{c,d} + x_{d,e} + x_{e,t}) 
\end{aligned}
\end{equation*}
Based on the graph symmetry, we might also expect an interdiction strategy of $c \rightarrow d$ and $c \rightarrow t$ to be optimal. When we run \nameEnumerateBinary\ on this terminating model: $\BMhatV$, we get three options with an optimal cost of 7 according to the \nameBM\  problem: 
\begin{enumerate}
    \item Interdicting $s \rightarrow c$ and $c \rightarrow d$
    \item Interdicting $s \rightarrow c$ and $d \rightarrow t$
    \item Interdicting $c \rightarrow d$ and $d \rightarrow t$
\end{enumerate}
We can check each of these three solutions with the true subproblem value and get that options 1 and 3 as expected are exact optimal solutions with a true cost of 7. Option 2 is not an exact optimal attack with a true cost of 5. Again by Theorem \ref{thm:proj_ef_subseteq_proj_benders}, we know that options 1 and 3 are the only exact optimal attacks.

When we look at alternative solutions in the defender response, we get a split in behavior depending on what is attacked. In option 1, the defender can respond by using $s \rightarrow a \rightarrow c \rightarrow d \rightarrow t$ or $s \rightarrow b \rightarrow c \rightarrow d \rightarrow t$. In option 3, the defender can respond by using $s \rightarrow c \rightarrow d \rightarrow e \rightarrow t$ or $s \rightarrow c \rightarrow d \rightarrow f \rightarrow t$.
The defender paths then only and always overlap in being forced to transit over $c \rightarrow d$.

\subsubsection{Attack 3 Arcs Example}
\label{sec:israeli_wood:examples:exp_3}

When the attacker may attack three arcs and the attack impacts are still all the same, we intuitively expect an interdiction strategy like the previous case that cancels out the symmetry by interdicting $s \rightarrow c$, $c \rightarrow d$, and $d \rightarrow t$. This is in fact an optimal solution at which our Benders Algorithm terminates with an optimal defender traversal cost of 8. 
When we then run \nameEnumerateBinary, the only option it returns is the strategy to interdict $s \rightarrow c$, $c \rightarrow d$, and $d \rightarrow t$. We now know that this is the only optimal solution by Theorem \ref{thm:proj_ef_subseteq_proj_benders}.
When we look at the defender response paths, we get alternative solutions there:
\begin{enumerate}
    \item $s \rightarrow a \rightarrow c \rightarrow d \rightarrow e \rightarrow t$
    \item $s \rightarrow b \rightarrow c \rightarrow d \rightarrow e \rightarrow t$
    \item $s \rightarrow a \rightarrow c \rightarrow d \rightarrow f \rightarrow t$
    \item $s \rightarrow b \rightarrow c \rightarrow d \rightarrow f \rightarrow t$
\end{enumerate}
The four solutions turn out to be all the valid paths from $s$ to $t$ that avoid both $s \rightarrow c$ and $d \rightarrow t$.

\subsection{Discussion}

Our \nameBenders\ approach generates alternative solutions for max-min style interdiction problems. 
On our specific problem instance, we get a sensitivity analysis that shows several things. 
 
First, $c \rightarrow d$ is always interdicted by a competent attacker. Second, the optimal defender response to optimal attack always changes based off attacker strength in this network. The first point is something that could be recognized without alternative solutions by looking for chokepoints. However, the second point leverages the ability to look at all of the exact optimal alternative solutions as attacker strength changes.

\section{Conclusion}
\label{sec:conclusions}

We demonstrate a generation method and structural theory for alternative solutions for Benders Decomposition. Our AOS-Benders process maintains the core feature \nameBenders\ is known for in trading a single problem over the entire variable space for multiple problems over subsets of variables. The AOS-Benders process can be appended to a traditional \nameBenders algorithm, subject to some technical assumptions, with existing alternative solution generation codes and a certification step. AOS-Benders also provides the capability to still generate alternative solutions when extensive form problems are either intractable or are otherwise undesirable. Such a capability is important on a range of problem classes including large-scale stochastic programming and max-min interdiction problems. We also have a theoretical characterization of the alternative solutions through the sublevel sets. This enables strong claims about first-stage solution properties including exhaustive solution enumeration under variable projection; without AOS-Benders, such claims would either require application-specific theory, require enumeration over full \nameEF solutions, or be intractable entirely. 

The new capabilities provided by our AOS-Benders algorithm raise several new questions about alternative solution generation. First, which technical assumptions made about \nameBenders can be relaxed (e.g. relatively complete recourse or continuous second-stage variables)? Second, can AOS-Benders improve the known scaling challenges \citep[e.g.][]{Lau_2024_mga_alternatives_energy_applications_review} of alternative solution generation? Third, while \citet{israeli_wood} used problem-specific methods to generate multiple subproblem solutions, could recent ML approximation methods for subproblems \citep[e.g.][]{Larsen_Lodi_Tactical_Solutions_IJOC} help with scaling and what specific cut-pool management rules would best fit the AOS-Benders paradigm?
Fourth, there are recent advances in stochastic and bilevel programming that rely on \nameBenders\ paradigms with some adaptation \citep[e.g.][respectively]{Ozgun_Hooker_Stochastic_Planning_Logic_Based_Benders_2022_IJOC, Byeon_Van_Hent_Benders_Bilevel_2022_IJOC}, raising which paradigms support AOS-Benders-like alternative solution generation methods?
Fifth, which other problem decomposition methods (e.g. Dantzig-Wolfe, Progressive Hedging) admit alternative solution generation methods and under what assumptions?

There are also a range of application areas for AOS-Benders that remain to be explored. The separable generation of alternative solutions has a range of applications from first-stage only generation in long term planning models (e.g. electrical grid capacity expansion) and generation of alternative solutions under privacy concerns. Given the variety of \nameBenders\ modifications and extensions, which of them adapt best to an alternative solutions paradigm (e.g. cut selection rules) remains an open question. All of these questions and applications stand to enhance the new optimization problems and modeling questions that can be treated by alternative solutions generally and under problem decomposition specifically.

\appendix
Here we treat all the proofs necessary to prove Theorem \ref{thm:proj_ef_subseteq_proj_benders} and the converse Remarks from Section \ref{sec:ext_benders:impact_of_approx}.
\section{Sublevel Set Technical Results}
We define two additional sublevel sets to match the intermediate steps from the \nameEF\ problem to the \nameBM\ problem made in Section \ref{sec:background:benders}.
\begin{align}
    \label{eqn:extending_benders:epigraphical_variant_level_set_def}
    \levelEV(\tau) = \levelset(&\funcG_{\theta}, \epigraph(\funcQ),\tau) \\
    \label{eqn:extending_benders:projected_variable_level_set_def}
    \levelPV(\tau) = \levelset(&\funcg+\funcQ, \setX,\tau) 
\end{align}

The following lemma shows that the relationship between the \nameEV\   problem and \nameBM\  problem sublevel sets is one of containment.
\begin{lemma}
\label{thm:benders_contains_EV_levelset}
    $\levelEV(\tau) \subseteq \levelBM(\tau), \forall \hatsetV \subseteq \setvpi, \forall \tau \in \mathR$
\end{lemma}
\begin{proof}{Proof}
This holds trivially if $\levelEV(\tau) = \varnothing$. Let $(\barvecx, \bartheta) \in \levelEV(\tau)$ for a given $\tau \in \mathR$. 
To establish $(\barvecx, \bartheta) \in \levelBM(\tau),\ \forall \hatsetV \subseteq \setvpi$ we need to show $\funcg(\barvecx) + \bartheta \leq \tau$ and $(\barvecx, \bartheta) \in \epigraph(\funchatQV), \ \forall \hatsetV \subseteq \setvpi$. We know that $\funcg(\barvecx) + \bartheta \leq \tau$ holds by definition of $\levelEV(\tau)$. Since $(\barvecx, \bartheta) \in \epigraph(\funcQ)$ by definition of $\levelEV(\tau)$, $(\barvecx, \bartheta) \in \epigraph(\funchatQV), \ \forall \hatsetV \subseteq \setvpi$ holds by Proposition \ref{thm:epigraph_containment}. So $(\barvecx, \bartheta) \in \levelBM(\tau),\ \forall \hatsetV \subseteq \setvpi$ and completes the containment proof. 
\end{proof}
The relationship between the \namePV\  problem and the \nameEV\   problem sublevel sets is equivalence when projected into the shared first-stage variables. The intuition here is that the \namePV\  problem and the \nameEV\   problem are effectively the same problem with $\theta$ serving as the helper variable.
\begin{lemma}
\label{thm:PV_is_projected_EV_levelset}
\(\levelPV(\tau) = \projX(\levelEV(\tau)), \forall \tau \in \mathR\)
\end{lemma}
\begin{proof}{Proof}
We prove this by double containment for fixed $\tau \in \mathR$. 
First, we show $\levelPV(\tau) \subseteq \projX(\levelEV(\tau))$. We take $\barvecx \in \levelPV(\tau)$ and let $\bartheta = \funcQ(\barvecx)$. We need to show that $(\barvecx, \bartheta) \in \epigraph(\funcQ)$ and $\funcg(\barvecx) + \bartheta \leq \tau$ to satisfy the definition of $\levelEV(\tau)$. Since $\barvecx \in \setX$ by definition of $\levelPV(\tau)$ and $\bartheta = \funcQ(\barvecx) \in \mathR$ by the relatively complete recourse and dual non-emptiness assumptions on $\funcQ$, this gives $(\barvecx, \bartheta) \in \epigraph(\funcQ)$. We know $\funcg(\barvecx) + \funcQ(\barvecx) \leq \tau$ by definition of $\levelPV(\tau)$, giving $\funcg(\barvecx) + \bartheta \leq \tau$. Thus $\barvecx \in \projX(\levelEV(\tau))$.
Second, we show that $\projX(\levelEV(\tau)) \subseteq \levelPV(\tau)$. We take $(\barvecx, \bartheta) \in \levelEV(\tau)$. We need to show that $\barvecx \in \setX$ and $\funcg(\barvecx) + \funcQ(\barvecx) \leq \tau$ to establish  $\barvecx \in \levelPV(\tau)$. We know $\barvecx \in X$ by definition of $\levelEV(\tau)$ as $\projX(Q) \subseteq \setX$. As $\funcg(\barvecx) + \bartheta \leq \tau$ and $\bartheta \geq \funcQ(\barvecx)$ by definition of $\levelEV(\tau)$, then $\funcg(\barvecx) + \funcQ(\barvecx) \leq \tau$ follows immediately. Thus $\barvecx \in \levelPV(\tau)$.
\end{proof}
The combination of Lemma \ref{thm:benders_contains_EV_levelset} and Lemma \ref{thm:PV_is_projected_EV_levelset} also gives the following corollary relating the \namePV\  problem and the \nameBM\  problem sublevel sets:
\begin{lemma}
\label{thm:PV_subset_proj_mp}
\(\levelPV(\tau) \subseteq \projX(\levelBM(\tau)), \forall \ \hatsetV \subseteq \setvpi, \forall \tau \in \mathR\)
\end{lemma}
\begin{proof}{Proof}
    This follows directly by the definition of the projection operator of \equationRef{eqn:proj_x_def} on Lemma \ref{thm:benders_contains_EV_levelset} and Lemma \ref{thm:PV_is_projected_EV_levelset}.
\end{proof}
The relationship between the \nameEF\  problem and the \namePV\  problem sublevel sets is equivalence once projected to the shared first-stage variables. The intuition here is that the \nameEF\  problem and the \namePV\  problem are solving the same problem with the value function wrapping the second-stage elements of the \nameEF\  problem.
\begin{lemma}
    \label{thm:proj_efp_is_PV}
    $\projX(\levelEF(\tau)) = \levelPV(\tau), \forall \tau \in \mathR$
\end{lemma}
\begin{proof}{Proof}
We prove this by double containment for a given $\tau \in \mathR$. 
First, we show that $\projX(\levelEF(\tau)) \subseteq \levelPV(\tau)$. We take $(\barvecx, \barvecy) \in \levelEF(\tau)$. We need to show that $\barvecx \in \setX$ and $\funcg(\barvecx) + \funcQ(\barvecx) \leq \tau$ to establish $\barvecx \in \levelPV(\tau)$. We know that $\barvecx \in \setX$ by definition of $\levelEF(\tau)$ as $\projX(\setGamma) \subseteq \setX$. We know that $\vecq^T\barvecy \geq \funcQ(\barvecx)$ for $(\barvecx, \barvecy) \in \setGamma$ by definition of $\funcQ$ and $\funcg(\barvecx) + \vecq^T\barvecy \leq \tau$ by definition of $\levelEF(\tau)$ giving $\funcg(\barvecx) + \funcQ(\barvecx) \leq \tau$. Thus $\barvecx \in \levelPV(\tau)$.
Second, we show that $\levelPV(\tau) \subseteq \projX(\levelEF(\tau))$. We take $\barvecx \in \levelPV(\tau)$ and choose $\barvecy \in \argmin_{\vecy \in \mathR^{n_2}_+} \vecq^T \vecy$ $\ s.t. \matrixW\vecy + \matrixT\vecx = \vech$ which is guaranteed to exist by the relatively complete recourse assumption. We need $(\barvecx, \barvecy) \in \setGamma$ and $\funcg(\barvecx) + \vecq^T\barvecy \leq \tau$ to show $\barvecx \in \projX(\levelEF(\tau))$. We know $\barvecx \in \setX$ by definition of $\levelPV(\tau)$ and $\barvecy$ was chosen to guarantee $\matrixW\barvecy + \matrixT\barvecx = \vech$ hence $(\barvecx, \barvecy) \in \setGamma$. We know that $\vecq^T\barvecy = \funcQ(\barvecx)$ by choice of $\barvecy$ and $\funcg(\barvecx) + \funcQ(\barvecx) \leq \tau$ by definition of $\levelPV(\tau)$ giving $\funcg(\barvecx) + \vecq^T\barvecy \leq \tau$. Thus $\barvecx \in \projX(\levelEF(\tau))$.
\end{proof}

\section{Non-Equivalence of the Extensive Form and Benders Sublevel Sets}
\label{sec:ext_benders:impact_of_approx:remarks}
In the previous section, we showed that the \nameBM\  problem sublevel set, for fixed $\tau$, contains the \nameEF\  sublevel set when both are projected to the shared first-stage variables. Ideally the two sublevel sets would be equivalent, which would enable the two sets to be used interchangeably. This is not the case in general as we show by providing a simple counterexample:
\begin{align*}
    \funcQ(x) = \min_{\vecy} &[1, \ \ 1] \, y\\
                        &[1, \ \ -1]\, \vecy = \vecx \\
                        &\vecy \in \mathR^2_+
\end{align*}
This corresponds to $\matrixW = [1, \ \ -1]$, $q^T= [1, \ \ 1]$, $T = -1$ and $\vech = 0$. This results in $\setPi = [-1, 1]$ and $\setvpi = \{-1,1\}$. The possible non-empty combination of vertices are as $\setV_1 = \{1\}$, $\setV_2 = \{-1\}$, and $\setV_3= \{-1,1\}$.

As approximations of $\funcQ$, we see $\funcQ_{\hatsetV_1}(x) = x$ and $\funcQ_{\hatsetV_2}(x) = -x$. We recover $\funcQ(x) = |x|$ with $\funcQ_{\hatsetV_3}$ using both dual vertices, which was shown visually in Figure \ref{fig:q_abs_x:full_figure}. We start with the relationship between approximate and exact epigraphs of the value function:
\begin{remark}
    $\epigraph(\funchatQV) \subseteq \epigraph (\funcQ)$ does not hold in general.
\end{remark}
\begin{proof}{Proof}
    Proof by construction.
    Let $\matrixW = [1, \ \ -1]$, $\vecq^T= [1, \ \ 1]$, $T = -1$ and $h = 0$. Then $\funcQ(x) = |x|$.
    Choose $\hatsetV$ = \{1\}, so $\funchatQV(x) = x$.
    $(-1,0) \in \epigraph(\funchatQV)$ but $(-1, 0) \not \in \epigraph(\funcQ)$.
\end{proof}
If the approximation of the value function is missing the supporting hyperplanes then there can be points admitted to the resulting approximate epigraph that the exact epigraph does not contain. Since the sublevel sets for the \nameBM\  problem rely on the approximate epigraph, $\epigraph(\funchatQV)$, it makes sense that the same example serves as a counterexample for the corresponding sublevel set relationships:
\begin{remark}
The converses of Lemma \ref{thm:benders_contains_EV_levelset}, Lemma \ref{thm:PV_subset_proj_mp}, and Theorem \ref{thm:proj_ef_subseteq_proj_benders} do not hold in general
\end{remark}
\begin{proof}{Proof}
    Proof by construction.
    Let $\matrixW = [1, \ \ -1]$, $\vecq^T= [1, \ \ 1]$, $T = -1$, $h = 0$ and and $\funcg(x) = 0$. Then $\funcQ(x) = |x|$ and $\funcg(x) + \funcQ(x) = |x|$.
    So $\levelEV(-1) = \varnothing$, $\levelPV(-1) = \varnothing$, and $\levelEF(-1) = \varnothing$. Let $\hatsetV$ = \{1\} then $\funchatQV(x) = x$ and $\funcg(x) + \funchatQV(x) = x$. Then $(-1,0) \in \levelBM(-1)$, but $(-1,0) \not\in \levelEV(-1)$, $\projX((-1,0)) = -1 \not\in \levelPV(-1)$, and $\projX((-1,0)) = -1, \ -1 \not\in \projX(\levelEF(-1))$.
\end{proof}
\bibliographystyle{informs2014}
\bibliography{bib_data}
\end{document}